\documentclass[preprint]{imsart}
\RequirePackage[OT1]{fontenc}
\usepackage{amsthm,amsmath,natbib,bm}
\RequirePackage[colorlinks,citecolor=blue,urlcolor=blue]{hyperref}
\usepackage{booktabs}

\arxiv{1004.0234}
\doi{10.1016/j.jspi.2013.01.007}

\startlocaldefs
\numberwithin{equation}{section}
\theoremstyle{plain}
\newtheorem{thm}{Theorem}[section]

\theoremstyle{remark}
\newtheorem{remark}{Remark}[section]

\endlocaldefs

\allowdisplaybreaks
\newcommand{\bep}{\bm{\epsilon}}
\newcommand{\bbe}{\bm{\beta}}
\newcommand{\byy}{\bm{y}}
\newcommand{\bX}{\bm{X}}
\newcommand{\bxx}{\bm{x}}
\newcommand{\bone}{\bm{1}}
\newcommand{\bzero}{\bm{0}}
\newcommand{\bth}{\bm{\theta}}

\begin{document}

\begin{frontmatter}
\title{Improved robust Bayes estimators of the error variance in linear models}
\runtitle{Improved robust Bayes Estimators}

\begin{aug}
\author{\fnms{Yuzo} \snm{Maruyama}
\thanksref{t1,m1}
\ead[label=e1]{maruyama@csis.u-tokyo.ac.jp}}
\and
\author{\fnms{William, E.} \snm{Strawderman}
\thanksref{t2,m2}
\ead[label=e2]{straw@stat.rutgers.edu}}

\thankstext{t1}{This work was partially supported by KAKENHI \#21740065 \& \#23740067.}
\thankstext{t2}{This work was partially supported by a grant from the Simons
Foundation (\#209035 to William Strawderman).}
\address{University of Tokyo\thanksmark{m1} and Rutgers University\thanksmark{m2} \\
\printead{e1,e2}}
\runauthor{Y. Maruyama and W. Strawderman}

\end{aug}

\begin{abstract}
We consider the problem of estimating the error variance in a general linear model
when the error distribution is assumed to be spherically symmetric, but not necessary
Gaussian. In particular we study the case of a scale mixture of Gaussians including
the particularly important case of the multivariate-$t$ distribution.
Under Stein's loss, we construct a class of estimators that improve
on the usual best unbiased (and best equivariant) estimator. Our class has the
interesting double robustness property of being simultaneously generalized Bayes
(for the same generalized prior) 
and minimax over the entire class of scale mixture of Gaussian distributions.
\end{abstract}


\end{frontmatter}

\section{Introduction}
\label{sec:intro}
Suppose the linear regression model is used to relate $y$ to the
$p$ predictors $x_1, \dots, x_p$,
\begin{equation} \label{full-model}
\byy = \alpha\bone_n+\bX\bbe + \sigma\bep
\end{equation}
where $\alpha$ is an unknown intercept parameter,
$\bone_n$ is an $n\times 1$ vector of ones,
$\bX=(\bxx_1,\dots, \bxx_p)$ is an $n \times p$ design matrix,
and $\bbe$ is a $p \times 1$ vector of unknown
regression coefficients.
In the error term, $\sigma$ is an unknown scalar and
 $\bep=(\epsilon_1,\dots,\epsilon_n)'$ 
 has a spherically symmetric distribution,
\begin{equation}\label{bep_sim_f}
\bep \sim f(\bep'\bep)
\end{equation}
where $f(\cdot)$ is the probability density,
$E[\bep]=\bzero_n$, 
and $\mbox{Var}[\bep]=\bm{I}_n$.
We assume that the columns of $\bX$ have been centered so that 
$\bxx'_i\bone_n = 0$ for $1 \leq i \leq p$.
We also assume that $n > p+1$ and $\{\bxx_1,\dots,\bxx_p\}$ are linearly independent, 
which implies that 
\begin{equation*}
\mbox{rank} \bX=p.
\end{equation*}
The class of error distributions we study includes the class of (spherical)
multivariate-$t$ distributions, probably the most important of the
possible alternative error distributions.
It is often felt in practice that the error distribution has heavier tails
than the normal and the class of multivariate-$t$ distributions is a flexible
class that allows for this possibility. They are also contained in the 
class of scale mixture of normal distributions and thus, by
 De Finetti's Theorem, represent exchangeable distributions regardless of the
sample size $n$.

In this paper we consider estimation of $\sigma^2=E[\{\sigma\epsilon_i\}^2]$,
the variance of each component of error term,
under Stein's loss (See \cite{James-Stein-1961}),
\begin{equation}\label{Stein'sLoss}
 L_S(\delta,\sigma^2)=\delta/\sigma^2-\log(\delta/\sigma^2)-1.
\end{equation}
Hence the risk function $ R(\{\alpha,\bbe,\sigma^2\},\delta)$
is given by $E[L_S(\delta,\sigma^2)]$.
The best equivariant estimator is the unbiased estimator given by
\begin{equation}
 \delta_U=\frac{\mbox{RSS}}{n-p-1}
\end{equation}
where RSS is Residual Sum of Squares given by
\begin{equation*}
 \mbox{RSS}=\|(I-\bX(\bX'\bX)^{-1}\bX')\{\byy-\bar{y}\bone_n\}\|^2.
\end{equation*}
In the Gaussian case, the Stein effect in the variance estimation problem
has been studied in many papers including
\cite{Stein-1964, Straw-1974a, Brewster-Zidek-1974, Maru-Straw-2006}.
\cite{Stein-1964} showed that
\begin{equation}
 \delta^{ST}=\min\left( \delta_U, \frac{\|\byy-\bar{y}\bone_n\|^2}{n-1}\right)
\end{equation}
dominates $ \delta_U$. For smooth (generalized Bayes) estimators,
\cite{Brewster-Zidek-1974} gave the improved estimator
\begin{equation*}
 \delta^{BZ}=\phi^{BZ}(R^2)\delta_U
\end{equation*}
where $ \phi^{BZ}(\cdot) $ is a smooth increasing function given by
\begin{equation} \label{phi-BZ}
 \phi^{BZ}(R^2)=1-\frac{2(1-R^2)^{(n-p-1)/2}}{n-1}
\left\{\int_0^1 t^{p/2-1}(1-R^2t)^{(n-p-1)/2}dt\right\}^{-1}
\end{equation}
and $R^2$ is the coefficient of determination given by 
\begin{equation}\label{eq:R^2}
R^2=\frac{ \|\bX(\bX'\bX)^{-1}\bX'\{\byy-\bar{y}\bone_n\}\|^2}{\|\byy-\bar{y}\bone_n\|^2}.
\end{equation}
\cite{Maru-Straw-2006} proposed another class of 
improved generalized Bayes estimators.
The proofs in all of these papers seem to depend strongly 
on the normality assumption. 
So it seems then, that it may be difficult or impossible to extend the dominance results to the 
non-normal case. 
Also many statisticians have thought that 
estimation of variance is more sensitive to the assumption of error distribution
compared to estimation of the mean vector, where some robustness results have been
derived by \cite{Maru-Straw-2005}.

Note that we use the term ``robustness''
in this sense of distributional robustness over the class of spherically
symmetric error distributions.
We specifically are not using the term to indicate a high breakdown point.
The use of the term ``robustness'' in our sense is however common
(if somewhat misleading) in the context of insensitivity to the
error distribution in the context of shrinkage literature.

In this paper, we derive a class of generalized Bayes estimators relative to
a class of separable priors of the form $\pi(\alpha,\bbe)\{\sigma^2\}^{-1}$
and show that the resulting generalized Bayes estimator is independent of the form
of the (spherically symmetric) sampling distribution. 
Additionally, we show, for a particular subclass of these separable priors, 
$ (\bbe'\bX'\bX\bbe)^{-(p-2)/2}\{\sigma^2\}^{-1}$, that the
resulting robust generalized Bayes estimator has the additional robustness property
of being minimax and dominating the unbiased estimator $\delta_U$ simultaneously,
for the entire class of scale mixture of Gaussians.

A similar (but somewhat stronger) robustness property has been studied in
the context of estimation of the vector of regression parameters $(\alpha,\bbe)$
by \cite{Maru-Straw-2005}. They gave separable priors of a form similar to
priors in this paper for which the generalized Bayes estimators are minimax
for the entire class of spherically symmetric distributions (and not just
scale mixture of normals).
We suspect that the distributional robustness property of the present paper
also extends well beyond the class of scale mixture of normal distributions
but have not been able to demonstrate just how much further it does extend.

%
The organization of this paper is as follows.
In Section \ref{sec:GB} we derive generalized Bayes estimators under separable
priors and demonstrate that the resulting estimator is independent of the (spherically symmetric)
sampling density.
In Section \ref{sec:minimax} we show that a certain subclass of estimators
which are minimax under normality remains minimax 
for the entire class of scale mixture of normals.
Further, we show that certain generalized Bayes estimators studied in
Section \ref{sec:GB} have this (double) robustness property. Some comments
are given in Section \ref{sec:CR} and an appendix
gives proofs of certain of the results.

\section{A generalized Bayes estimator with respect to the harmonic prior}
\label{sec:GB}
In this section, we show that 
the generalized Bayes estimator of the variance
with respect to a certain class of priors
is independent of the particular sampling model under Stein's loss.
Also we will give an exact form of this estimator for a particular subclass of 
``(super)harmonic'' priors that, we will later show, is minimax for a large
subclass of spherically symmetric error distributions.
\begin{thm} \label{thm:indep}
The generalized Bayes estimator with respect to 
$\pi(\alpha, \bbe,\sigma^2)=\pi(\alpha,\bbe)\{\sigma^2\}^{-1}$
under Stein's loss \eqref{Stein'sLoss} is independent of the particular 
spherically symmetric sampling model and hence
is given by the generalized Bayes estimator under the Gaussian distribution.
\end{thm}
\begin{proof}
 See Appendix.
\end{proof}

Now let $p \geq 3$ and $ \pi(\alpha,\bbe)= (\bbe'\bX'\bX\bbe)^{-(p-a)/2}$.
This is related to a family of (super)harmonic functions as follows.
If, in the above joint prior for $(\alpha,\bbe)$, we make 
the change of variables, $\bth=(\bX'\bX)^{1/2}\bbe$,
the joint prior of $(\alpha,\bth)$ becomes
\begin{equation} \label{improper-joint-1}
\pi(\alpha,\bth) 
 = \|\bth\|^{-(p-a)}.
\end{equation}
The Laplacian of $\|\bth\|^{-(p-a)}$ 
is given by 
\begin{equation*}
 \sum_{i=1}^p \frac{\partial^2}{\partial \theta_i^2}
\|\bth\|^{-(p-a)} 
 = (p-a)(2-a)\|\bth\|^{-(p-a)-2},
\end{equation*}
which is negative (i.e.~super-harmonic) for $2<a<p$ and is zero (i.e.~harmonic)
for $a=2$. 
\begin{thm}\label{thm:harmonic}
Under the model \eqref{full-model} with spherically symmetric error distribution 
\eqref{bep_sim_f}
and Stein's loss \eqref{Stein'sLoss}, 
the generalized Bayes estimator with respect to 
$\pi(\alpha,\bbe,\sigma^2)=(\bbe'\bX'\bX\bbe)^{-(p-a)/2}\{\sigma^2\}^{-1}$ 
for $0<a<p $ is given by
\begin{equation} \label{eq:GB}
 \delta_a^{GB}=\phi_a^{GB}(R^2)\frac{\mathrm{RSS}}{n-p-1}
\end{equation}
where 
\begin{equation}
\phi_a^{GB}(R^2)= \frac{n-p-1}{n-a-1}\frac{\int_0^1
t^{p/2-a/2-1}(1-t)^{a/2-1}(1-R^2t)^{(n-p-a-1)/2}
\, dt}{\int_0^1
t^{p/2-a/2-1}(1-t)^{a/2-1}(1-R^2t)^{(n-p-a+1)/2}
\, dt}.
\end{equation}
\end{thm}
\begin{proof}
See Appendix.
\end{proof}

\section{Minimaxity}
\label{sec:minimax}
In this section, we demonstrate robustness of 
minimaxity under scale mixture of normals for a class of estimators which 
are minimax under normality.
\begin{thm}\label{thm:main}
Assume
$\delta_\phi=\phi(R^2)\{\mathrm{RSS}/(n-p-1)\}$ where $\phi(\cdot)$ is monotone 
nondecreasing, improves on the unbiased estimator, $\delta_U$, under normality
and Stein's loss.
Then $\delta_\phi$ also improves on the unbiased estimator,
$\delta_U$, under scale mixture of normals and Stein's loss.
\end{thm}
\begin{proof}
Let $f$ be a scale mixture of normals where the scalar $\tau$ 
satisfies $E[\tau^2]=1$, that is, 
\begin{equation*}
 f(t)=\int_0^\infty (2\pi\tau)^{-n/2}\exp(-t/\{2\tau^2\})g(\tau^2)d\tau^2.
\end{equation*}
Then $\bm{y}|\tau^2\sim N_n(\alpha\bone_n+\bX\bbe,\sigma^2\tau^2\bm{I}_n)$ and
the risk difference between these estimators is given by
\begin{equation}\label{eq:risk-diff}
\begin{split}
&R\left(\{\alpha,\bbe,\sigma^2\}, \delta_U\right)
-R\left(\{\alpha,\bbe,\sigma^2\}, \delta_\phi\right) \\
&= E_{\byy}\left[ \frac{\left\{1-\phi(R^2)\right\}}{n-p-1} \frac{\mbox{RSS}}{\sigma^2}
+\log \phi(R^2) \right] \\
&= E_{\tau^2}\left[ E_{\byy|\tau^2}\left[ \frac{\left\{1-\phi(R^2)\right\}}{n-p-1} \frac{\mbox{RSS}}{\sigma^2}
+\log \phi(R^2) \right] \right]\\
 & = E_{\tau^2}\left[E_{\byy|\tau^2}
\left[ \frac{\left\{1-\phi(R^2)\right\}\mbox{RSS}}{(n-p-1)\{\sigma^2\tau^2\}}
+\log \phi(R^2) \right] \right]  \\
&\qquad+\frac{1}{n-p-1}E_{\tau^2}\left[E_{\byy|\tau^2}\left[\frac{\{1-\phi(R^2)\}\mbox{RSS}}{\sigma^2\tau^2}
\left(\tau^2-1\right)\right]\right] .
\end{split}
\end{equation}
In the the first term of the right-hand side of the above equality,
\begin{equation*}
 E_{\byy|\tau^2}
\left[ \frac{\left\{1-\phi(R^2)\right\}\mbox{RSS}}{(n-p-1)\{\sigma^2\tau^2\}}
+\log \phi(R^2) \right]
\end{equation*}
is the risk difference under the Gaussian assumption, which is given by
\begin{equation*}
R\left(\{\alpha,\bbe,\tau^2\sigma^2\}, \delta_U\right)
-R\left(\{\alpha,\bbe,\tau^2\sigma^2\}, \delta_\phi\right)
\end{equation*}
where $ \bm{y}\sim N_n(\alpha\bone_n+\bX\bbe,\sigma^2\tau^2\bm{I}_n)$.
From the assumption of the theorem, it is non-negative for any $\tau^2>0$. 
Hence it suffices to show that the second term is non-negative.

For given $\tau^2$, $ \{\|\byy- \bar{y}\bone_n\|^2-\mbox{RSS}\}/\{\sigma^2\tau^2\}(=U)$ 
and $\mbox{RSS}/\{\sigma^2\tau^2\}(=V)$ are independently distributed 
as $ \chi^2_{p}(\lambda/\tau^2)$
with $\lambda=\bbe'\bX'\bX\bbe/\sigma^2$ and $ \chi^2_{n-p-1}$.
Since $R^2$ is given by $1-\mbox{RSS}/\|\byy-\bar{y}\bone_n\|^2=(1+V/U)^{-1}$,
the second term of the right-hand side of \eqref{eq:risk-diff}
is written as
\begin{equation*}
\begin{split}
&E_{\tau^2}\left[E_{\byy|\tau^2}\left[\frac{\{1-\phi(R^2)\}\mbox{RSS}}{\sigma^2\tau^2}
\left(\tau^2-1\right)\right]\right]  \\
&=E_{\{U,V,\tau^2\}}\left[\left\{1-\phi(\{1+V/U\}^{-1})\right\} (\tau^2-1) V\right]  \\
& =E_{V}\left[E_{\tau^2|V}\left[\psi(\tau^2,V)(\tau^2-1)\right]V\right],
\end{split}
\end{equation*}
where 
\begin{equation}
 \psi(\tau^2,v)=1-E\left[ \phi(\{1+v/\chi^2_p(\lambda/\tau^2)\}^{-1}) \right].
\end{equation}
By the monotone likelihood ratio property of non-central $\chi^2$,
$\psi(\tau^2,v)$ is non-decreasing in $\tau^2$ for any fixed $v$.
Further, by the covariance inequality, 
 \begin{equation}\label{eq:inequality}
 E_{\tau^2|V}\left[\psi(\tau^2,V)(\tau^2-1)\right] 
  \geq  E_{\tau^2|V}[\tau^2-1]  E_{\tau^2|V}\left[\psi(\tau^2,V) \right] 
=0
 \end{equation}
since $V$ and $\tau^2$ are mutually independent and $E[\tau^2]=1$.
The inequality \eqref{eq:inequality} implies that
the second term of the right-hand side of \eqref{eq:risk-diff} is non-negative.
\end{proof}
Under the normality assumption, \cite{Brewster-Zidek-1974} showed that
the estimator $\phi(R^2)\delta_U$ with nondecreasing $ \phi$
dominates the unbiased estimator $ \delta_U$ if $ \phi^{BZ} \leq \phi \leq 1$,
where $\phi^{BZ}$ is given by \eqref{phi-BZ}.
\cite{Maru-Straw-2006} demonstrated that the generalized Bayes estimator 
of Theorem \ref{thm:harmonic} with $a=2$ satisfies this condition.
Hence our main result shows that the generalized Bayes estimator 
of Theorem \ref{thm:harmonic} with $a=2$,
is minimax for the entire class of variance mixture of normal distributions.
\begin{thm}\label{thm:main2}
Let $ n-1> p\geq 3$. Under Stein's loss, the estimator given by
\begin{equation}\label{delta-H}
\delta^{H}=\phi^{H}(R^2)\frac{\mathrm{RSS}}{n-p-1}
\end{equation}
where 
\begin{equation}\label{phi-H}
\phi^{H}(R^2)= \frac{n-p-1}{n-3}\frac{\int_0^1
t^{p/2-2}(1-R^2t)^{(n-p-3)/2}
\, dt}{\int_0^1
t^{p/2-2}(1-R^2t)^{(n-p-1)/2}
\, dt}
\end{equation}
is minimax and generalized Bayes with respect to the harmonic prior
\begin{equation} \label{harmonic-prior}
 \pi(\alpha,\bbe,\sigma^2)=(\bbe'\bX'\bX\bbe)^{-(p-2)/2}\{\sigma^2\}^{-1}
\end{equation}
for the entire class of scale mixture of normals.
\end{thm}

\begin{remark}\label{rem:Rsq}
Note that the coefficient of determination is given in \eqref{eq:R^2}
and
that the expectations of the numerator and the denominator are given by
\begin{equation*}
\begin{split}
& E\left[\|\bX(\bX'\bX)^{-1}\bX'\{\byy-\bar{y}\bone_n\}\|^2\right]=
\sigma^2\{\xi+p\}, \\
& E\left[\|\byy-\bar{y}\bone_n\|^2\right]=\sigma^2\{\xi+n-1\},
\end{split}
\end{equation*}
where $\xi=\bbe'\bX'\bX\bbe/\sigma^2$.
Hence the smaller $R^2$ corresponds to the smaller $ \xi$ since $n-1>p$.

Our class of improved estimators utilizes the coefficient of determination
$R^2$ in making a (smooth) choice between $\delta_U$ (when $R^2$ and $\xi$ are large)
and $\|\byy-\bar{y}\bone_n\|^2/(n-1)$ (when $R^2$ and $\xi$ are small)
and reflects the relatively common knowledge among statisticians, 
that $\|\byy-\bar{y}\bone_n\|^2/(n-1)$
is stochastically closer to $\sigma^2$ when $R^2$ is small.
\end{remark}

\begin{remark}
The estimator $\delta^{H}$ is not the only minimax generalized Bayes estimator
under scale mixture of normals. 
In Theorem \ref{thm:harmonic}, we also provided the generalized Bayes estimator   
with respect to superharmonic prior given by  
$\pi(\alpha,\bbe,\sigma^2)=(\bbe'\bX'\bX\bbe)^{-(p-a)/2}\{\sigma^2\}^{-1}$. 
In \cite{Maru-Straw-2006}, we show that for $\delta^{GB}_a$ 
with $2< a(n,p) \leq a <p$ is minimax in the normal case with a monotone $\phi_a^{GB}$.
Hence for $a$ in this range $\delta^{GB}_a$ is also minimax and generalized Bayes
for the entire class of scale mixture of normals.
The bound $a(n,p)$ has a somewhat complicated form 
and we omit the details (however, see \cite{Maru-Straw-2006} for details).

Note that $\delta^H$ corresponds to $\delta^{GB}_a$ with $ a=2$
since the corresponding prior to $\delta^H$  is given by \eqref{harmonic-prior}.
Note also that the unbiased estimator $\delta_U$ which is derived as
the Jeffrey's prior $ \pi(\alpha,\bbe,\sigma^2)=1/\sigma^{2}$ corresponds
to $\lim_{a\to p}\delta^{GB}_a$.
Therefore
we conjecture that $\delta^{GB}_a$ with any $a\in(2,p)$ is minimax.
\end{remark}

\begin{remark}
Under the normality assumption, 
\cite{Maru-Straw-2006} gave a subclass of minimax 
generalized Bayes estimators with the particularly simple form
\begin{equation}\label{delta^SB}
\delta^{SB}=\left\{1+c (1-R^2)\right\}^{-1}\frac{\mbox{RSS}}{n-p-1}
\end{equation}
for $0<c \leq c(n,p)$ where $c(n,p)$ has a slightly complicated form, which we omit
(see \cite{Maru-Straw-2006} for details).
Under spherical symmetry, this estimator is not necessarily derived 
as generalized Bayes (See the following Remark), 
but is still minimax under scale mixture of normals.
\end{remark}
\begin{remark}
Interestingly, when $ (n-1)/2<p<(n-1)$, the generalized Bayes estimator
with respect to $\|\bbe'\bX'\bX\bbe\|^{-p+(n-1)/2}\{\sigma^2\}^{-1}$ is given by
\begin{equation}\label{delta_*^SB}
\delta^{SB}_*= \left(1+ \frac{2p-n+1}{n-p-1} (1-R^2)\right)^{-1}
\frac{\mbox{RSS}}{n-p-1}
\end{equation}
for the entire class of spherically symmetric distributions 
(See \cite{Maru-Straw-2006} for the technical details).
Hence when 
\begin{equation}\label{cond:n.p}
 (2p-n+1)/(n-p-1) \leq c(n,p),
\end{equation}
$\delta^{SB}_*$ is minimax and generalized Bayes for the entire class of
scale mixture of normals. 
Unfortunately, numerical calculations indicate that, 
for $n$ in the range $(25, 10,000)$, the inequality \eqref{cond:n.p}
is only satisfied 
for $p= (n+1)/2$ for $n$ odd and $n/2$ and $n/2 +1$ for $n$ even.

Actually, under the Gaussian assumption,
 $\delta^{SB}$ given in \eqref{delta^SB} with $c$ larger than $c(n,p)$
can be demonstrated to be minimax numerically even though 
our analytic upper bound on $c$ for minimaxity is $c(n,p)$.
In practice, since $ \phi^H$ given in \eqref{phi-H} can be calculated
quickly and precisely, we recommend the use of $\delta^H$ given in \eqref{delta-H}.

\end{remark}

\begin{remark}
For Theorems \ref{thm:indep}, \ref{thm:main} and \ref{thm:main2}, 
the choice of the loss function is the key.
Many of the results introduced in Section \ref{sec:intro}
were initially proved under the quadratic loss function $(\delta/\sigma^2-1)^2$.
Under the Gaussian assumption, the corresponding results can be obtained by
replacing $n+2$ by $n$. On the other hand, the generalized Bayes estimator
with respect to $\pi(\alpha, \bbe,\sigma^2)=\pi(\alpha,\bbe)\{\sigma^2\}^{-1}$
depends on the particular sampling model 
and hence robustness results do not hold under non-Gaussian assumption. 
\end{remark}

\section{Concluding Remarks}
\label{sec:CR}
In this paper, we have studied estimation of the error variance in a general linear model
with a spherically symmetric error distribution.
We have shown, under Stein's loss, that
separable priors of the form $\pi(\alpha,\bbe)\{\sigma^2\}^{-1}$
have associated generalized Bayes estimators which are independent of the form
of the (spherically symmetric) sampling distribution. 
We have further exhibited a subclass of ``superharmonic'' priors for which
these generalized Bayes estimators dominate the usual unbiased and best
equivariant estimator, $\delta_U$, for the entire class of scale mixture
of normal error distributions.

We have previously studied a very similar class of prior distributions in the problem
of estimating the regression coefficients $(\alpha,\bbe)$ under quadratic loss
(See \cite{Maru-Straw-2005}).
In that study we demonstrated a similar double robustness property: to wit,
that the generalized Bayes estimators are independent of the form of the
sampling distribution and that they are minimax over the entire class of
spherically symmetric distributions.

The main difference between the classes of priors in the two settings are
a) in the present study, the prior on $\sigma^2$ is proportional to $\{\sigma^2\}^{-1}$
while it is proportional to $ \{\sigma^2\}^{a}$ in the earlier study; and b)
in this paper, the prior on $(\alpha,\bbe)$ is also separable with $\alpha$
being uniform on the real line and $\bbe$ having the ``superharmonic'' form,
while in the earlier paper $(\alpha,\bbe)$ jointly had the superharmonic form.

The difference a) is essential since a prior on $\sigma^2$ proportional to
$\{\sigma^2\}^{-1}$ gives the best equivariant and minimax estimator $\delta_U$,
while such a restriction is not necessary when estimating the regression 
parameters $(\alpha,\bbe)$.

The difference in b) is inessential, and either form of priors on the
regression parameters $(\alpha,\bbe)$ will give estimators with the
double robustness properties in each of the problems studied.
The form of the estimators, of course, will be somewhat different.
In the case of the present paper, the main difference would be to replace
$n-p-1$ by $n-p$ and to replace $R^2$ by
\begin{equation*}
 \{n\bar{y}^2+\|\bX(\bX'\bX)^{-1}\bX'\byy\|^2\}/\|\byy\|^2.
\end{equation*}

As a consequence, the results in these papers suggest that separable priors,
and in particular the ``harmonic'' prior given \eqref{harmonic-prior}, are very worthy
candidates as objective priors in regression problems.
They produce generalized Bayes minimax procedures dominating the classical
unbiased, best equivariant estimators of both regression parameters and 
scale parameters simultaneously and uniformly over a broad class of 
spherically symmetric error distributions.

\appendix
\section{Proof of Theorem \ref{thm:indep}}
The (generalized) Bayes estimator with Stein's loss
is given by $\{E[1/\sigma^2|\byy]\}^{-1}$.  
Under the improper density $\pi(\alpha,\bbe,\sigma^2)=\pi(\alpha,\bbe)\{\sigma^2\}^{-1}$, the generalized Bayes estimator is given by
\begin{equation*}
 \frac{\iint m^f_0(\byy|\alpha,\bbe)\pi(\alpha,\bbe)d\alpha d\bbe}
{\iint m^f_1(\byy|\alpha,\bbe)\pi(\alpha,\bbe)d\alpha d\bbe} 
\end{equation*}
where $m^f_i(\byy|\alpha,\bbe)$ for $i=0,1$ is
the conditional marginal density of $\byy$ with respect to 
$\{\sigma^2\}^{-1-i}$ given $\alpha$ and $\bbe$,
\begin{equation*}
 m^f_i(\byy|\alpha,\bbe)=\int_0^\infty \sigma^{-n} 
f\left(\frac{\|\byy-\alpha \bone_n -\bX\bbe\|^2}{\sigma^2}\right)
(\sigma^2)^{-i-1} d\sigma^2.
\end{equation*}
Further we have
\begin{equation*}
\begin{split}
&  m^f_i(\byy|\alpha,\bbe)
 = \|\byy-\alpha \bone_n -\bX\bbe\|^{-n-2i}\int_0^\infty t^{\{n+2i\}/2-1}f(t)dt  \\
&= \frac{\int_0^\infty t^{(n+2i)/2-1}f(t)dt}{\int_0^\infty t^{(n+2i)/2-1}f_G(t)dt}
\int_0^\infty 
f_G\left(\frac{\|\byy-\alpha \bone_n -\bX\bbe\|^2}{\sigma^2}\right)
\frac{(\sigma^2)^{-i-1}}{\sigma^{n} } d\sigma^2 
\end{split}
\end{equation*}
where 
\begin{equation*}
 f_G(t)=\frac{1}{(2\pi)^{n/2}}\exp(-t/2).
\end{equation*}
Hence the generalized Bayes estimator 
is
\begin{equation*}
 \frac{\iint m^f_0(\byy|\alpha,\bbe)\pi(\alpha,\bbe)d\alpha d\bbe}
{\iint m^f_1(\byy|\alpha,\bbe)\pi(\alpha,\bbe)d\alpha d\bbe} 
 =  \frac{\int_0^\infty t^{n/2-1}f(t)dt}{\int_0^\infty t^{n/2}f(t)dt}
\frac{\int_0^\infty t^{n/2}f_G(t)dt}{\int_0^\infty t^{n/2-1}f_G(t)dt}
\frac{m_0^G(\byy)}{m_1^G(\byy)}
\end{equation*}
where
\begin{equation*}
 m_i^G(\byy)=\iiint
f_G\left(\frac{\|\byy-\alpha \bone_n -\bX\bbe\|^2}{\sigma^2}\right)
\frac{\pi(\alpha,\bbe)}{(\sigma^2)^{n/2+i+1} }d\alpha d\bbe  d\sigma^2.
\end{equation*}
Since $\bep$ has a spherically symmetric density $f(\bep'\bep)$ and
$E[\bep]=\bzero_n$ and $\mathrm{Var}[\bep]=\bm{I}_n$,
$f$ as well as $f_G$ satisfies 
\begin{equation} \label{f-density}
\int_{\mathcal{R}^n}f( \bep'\bep)d\epsilon
=\frac{\pi^{n/2}}{\Gamma(n/2)}\int_0^\infty s^{n/2-1}f(s)ds=1,
\end{equation}
and
\begin{equation} \label{f-density-1}
\int_{\mathcal{R}^n}\bep'\bep f( \bep'\bep)d\epsilon
=
\frac{\pi^{n/2}}{\Gamma(n/2)}\int_0^\infty s^{n/2}f(s)ds=n.
\end{equation}
Hence we have
\begin{equation*}
 \frac{\int_0^\infty t^{n/2-1}f(t)dt}{\int_0^\infty t^{n/2}f(t)dt}
\frac{\int_0^\infty t^{n/2}f_G(t)dt}{\int_0^\infty t^{n/2-1}f_G(t)dt}
=\frac{1}{n} \cdot \frac{n}{1}=1
\end{equation*}
and hence the generalized Bayes estimator 
is given by $ m_0^G(\byy)/m_1^G(\byy)$ which is independent of $f$.

\section{Proof of Theorem \ref{thm:harmonic}}
Note, for $0<a<p$, 
\begin{equation} \label{improper-joint}
\begin{split}
& (\bbe'\bX'\bX\bbe)^{-(p-a)/2}
= \frac{2^{a/2}\pi^{p/2}}{\Gamma(\{p-a\}/2)|\bX'\bX|^{1/2} } \\
& \qquad  \qquad  \qquad \times 
\{\sigma^2\}^{a/2}\int_0^{\infty} g^{a/2-1} 
\frac{|\bX'\bX|^{1/2}}{(2\pi\sigma^2g)^{p/2}}
\exp\left(-\frac{\bbe'\bX'\bX\bbe}{2\sigma^2 g}\right)dg.
\end{split}
\end{equation}
Then
\begin{equation} \label{m-i-G}
\begin{split}
& m^G_i(\byy) 
=A 
\int_{-\infty}^{\infty} \int_{R^{p}}
 \int_{0}^{\infty}  \int_{0}^{\infty}
\frac{1}{(2\pi\sigma^2)^{n/2}} 
\exp\left(-\frac{\|\byy-\alpha \bone_n -\bX\bbe\|^2}{2\sigma^2}\right) \\
& \qquad \times  
\frac{g^{a/2-1}}{\{\sigma^2\}^{-a/2+1+i}}
\frac{|\bX'\bX|^{1/2}}{(2\pi\sigma^2)^{p/2}g^{p/2}}
\exp\left(-\frac{\bbe'\bX'\bX\bbe}{2\sigma^2 g}\right)
d \alpha \, d \bbe \, d\sigma^2 \, dg,
\end{split}
\end{equation}
where $ A=\{2^{a/2}\pi^{p/2}\}/\{\Gamma(\{p-a\}/2)|\bX'\bX|^{1/2} \}$.
In the following, we calculate the integration in \eqref{m-i-G}
with respect to $\alpha$, $\bbe$, 
$\sigma^2$, and $g$, in this order.

By the simple relation
\begin{equation*}
\byy -\alpha \bone_n-\bX \bbe
= (-\alpha+\bar{y})\bone_n
 + \bm{v}-\bX \bbe
\end{equation*}
where $\bar{y}$ mean the mean of $\byy$ and $\bm{v}=\byy-\bar{y}\bone_n$,
we have the Pythagorean relation,
\begin{equation*}
\| \byy -\alpha \bone_n-\bX \bbe \|^2
= n(-\alpha+\bar{y})^2
 + \|\bm{v}-\bX \bbe \|^2,
\end{equation*}
since $\bX$ has been already centered.
Then we have 
\begin{equation*}
\begin{split}
&\int_{-\infty}^{\infty} 
\frac{1}{(2\pi\sigma^2)^{n/2}}
\exp\left(-\frac{\|\byy-\alpha \bone_n -\bX\bbe\|^2}{2\sigma^2}\right)
d \alpha  \\
&=
 \frac{n^{1/2}}{(2\pi\sigma^2)^{(n-1)/2}}
\exp\left( 
-\frac{\| \bm{v}-\bX\bbe\|^2}{2\sigma^2}\right) . 
\end{split}
\end{equation*}
Next we consider the integration with respect to $\bbe$.
Note the relation of completing squares with respect to $\bbe$ 
\begin{equation*}
\begin{split}
& \| \bm{v}-\bX\bbe\|^2+g^{-1}\bbe'\bX'\bX\bbe \\
&= \frac{1+g}{g}\left(\bbe - \frac{g}{1+g}\hat{\bbe}\right)'\bX'\bX
\left(\bbe - \frac{g}{1+g}\hat{\bbe}\right) 
+ \frac{\|\bm{v}\|^2}{1+g}\left\{ g(1-R^2)+1\right\}
\end{split}
\end{equation*}
where $\hat{\bbe}=(\bX'\bX)^{-1}\bX'\bm{v}$ and 
$R^2=\|\bX\hat{\bbe}\|^2/\|\bm{v}\|^2$ is the coefficient of determination.
Hence we have
\begin{equation}\label{marginal-known-3}
\begin{split}
&\int_{-\infty}^{\infty} \int_{R^p}\frac{1}{(2\pi\sigma^2)^{n/2}} 
\exp\left(-\frac{\|\byy-\alpha \bone_n -\bX\bbe\|^2}{2\sigma^2}\right) \\
& \quad \times \frac{|\bX'\bX|^{1/2}}{(2\pi\sigma^2)^{p/2}g^{p/2}}
\exp\left(-\frac{\bbe'\bX'\bX\bbe}{2\sigma^2 g}\right)
d \alpha \, d\bbe  \\
& \qquad =  \frac{n^{1/2}(1+g)^{-p/2}}{(2\pi\sigma^2)^{(n-1)/2}}
\exp\left(
-\frac{\|\bm{v}\|^2\{g(1-R^2)+1\}}{2\sigma^2(g+1)} \right).
\end{split}
\end{equation}
Next we consider integration with respect to $\sigma^2$.
By \eqref{marginal-known-3}, we have
\begin{equation}\label{marginal-known-4}
\begin{split}
&\int_{-\infty}^{\infty} \int_{R^p} \int_0^\infty
\frac{1}{(2\pi\sigma^2)^{n/2}}
\exp\left(-\frac{\|\byy-\alpha \bone_n -\bX\bbe\|^2}{2\sigma^2}\right) \\
& \quad \times \frac{|\bX'\bX|^{1/2}}{(2\pi\sigma^2)^{p/2}g^{p/2}}
\exp\left(-\frac{\bbe'\bX'\bX\bbe}{2\sigma^2 g}\right) \{\sigma^2\}^{a/2-1-i}
d \alpha \, d\bbe \, d\sigma^2\\
& \qquad =  \frac{2^{-a/2+i}n^{1/2}\Gamma(\{n-a-1+2i\}/2)}{\pi^{(n-1)/2}\|\bm{v}\|^{n-a-1+2i}}
\frac{(1+g)^{(n-p-a-1+2i)/2}}{\left\{ g(1-R^2)+1\right\}^{(n-a-1+2i)/2}}.
\end{split}
\end{equation}
Finally  we consider integration with respect to $g$. 
By \eqref{marginal-known-4} we have
\begin{equation}
\begin{split}
 m_i^G(\byy) 
& = 
A\frac{2^{-a/2+i}n^{1/2}\Gamma(\{n-a-1+2i\}/2)}{\pi^{(n-1)/2}\|\bm{v}\|^{n-a-1+2i}} \\
&   \qquad \times \int_0^\infty 
\frac{g^{a/2-1}(1+g)^{(n-p-a-1+2i)/2}}{\left\{ g(1-R^2)+1\right\}^{(n-a-1+2i)/2}}
\, dg \\
& = A\frac{2^{-a/2+i}n^{1/2}\Gamma(\{n-a-1+2i\}/2)}{\pi^{(n-1)/2}\|\bm{v}\|^{n-a-1+2i}
(1-R^2)^{(n-p-1-2i)/2}}
 \\
&  \qquad \times \int_0^1
t^{p/2-a/2-1}(1-t)^{a/2-1}(1-R^2t)^{(n-p-a-1+2i)/2}
\, dt.
\end{split}
\end{equation}
The second equality follows from the change of variables $1/\{1+g(1-R^2)\} \to t$.
By using the relation $ (1-R^2)\|\byy-\bar{y}\bone_n\|^2=\mbox{RSS}$, 
$m_0^G(\byy)/m_1^G(\byy)$ is written as \eqref{eq:GB}.

\section*{Acknowledgements}
We are very grateful to the associate editor, and referees 
for wonderful insights which substantially helped us to strengthen this paper.


\end{document}